\documentclass[11pt,oneside,english]{amsart}
\usepackage[T1]{fontenc}
\usepackage[utf8]{inputenc}
\usepackage{color}
\usepackage{babel}
\usepackage{amstext}
\usepackage{amsthm}
\usepackage{amssymb}
\usepackage{graphicx}
\usepackage[caption=false]{subfig}
\usepackage[unicode=true,pdfusetitle,
 bookmarks=true,bookmarksnumbered=false,bookmarksopen=false,
 breaklinks=false,pdfborder={0 0 0},backref=false,colorlinks=true]
 {hyperref}
\hypersetup{
 citecolor=blue}
\theoremstyle{Theorem A}
\theoremstyle{Theorem B}
\theoremstyle{Theorem C}
\theoremstyle{Theorem D}
\theoremstyle{Theorem E}

\newtheorem*{thmA}{Theorem A}
\makeatletter
\numberwithin{equation}{section}
\numberwithin{figure}{section}
\theoremstyle{plain}
\newtheorem*{cor*}{\protect\corollaryname}
\theoremstyle{plain}
\newtheorem{thm}{\protect\theoremname}[section]
\theoremstyle{definition}
\newtheorem{defn}[thm]{\protect\definitionname}
\theoremstyle{question}

\theoremstyle{remark}
\newtheorem{rem}[thm]{\protect\remarkname}
\theoremstyle{plain}
\newtheorem{prop}[thm]{\protect\propositionname}
\theoremstyle{plain}
\newtheorem{lem}[thm]{\protect\lemmaname}
\theoremstyle{plain}
\newtheorem{cor}[thm]{\protect\corollaryname}

\usepackage{geometry}
\usepackage{amsmath}

\numberwithin{equation}{section}
\numberwithin{figure}{section}
\usepackage{enumitem}		% customizable list environments
 \let\footnote=\endnote
\@ifundefined{lettrine}{\usepackage{lettrine}}{}
\theoremstyle{definition}

\def\R{\mathbb{R}}

\def\N{\mathbb{N}}
\def\Z{\mathbb{Z}}

\usepackage{float}

\keywords{}

\subjclass[2000]{}

\def\R{\mathbb{R}}

\def\loc{\text{loc}}

\def\A{\mathcal{A}}

\def\glr{\text{GL}_d(\R)}

\makeatother

  \providecommand{\corollaryname}{Corollary}
  \providecommand{\definitionname}{Definition}
  \providecommand{\lemmaname}{Lemma}
  \providecommand{\propositionname}{Proposition}
  \providecommand{\remarkname}{Remark}
  \providecommand{\theoremname}{Theorem}
\providecommand{\theoremname}{Theorem}

\usepackage{tikz,xcolor,hyperref}

\definecolor{lime}{HTML}{A6CE39}
\DeclareRobustCommand{\orcidicon}{
	\begin{tikzpicture}
	\draw[lime, fill=lime] (0,0) 
	circle [radius=0.16] 
	node[white] {{\fontfamily{qag}\selectfont \tiny ID}};
	\draw[white, fill=white] (-0.0625,0.095) 
	circle [radius=0.007];
	\end{tikzpicture}
	\hspace{-2mm}
}
\foreach \x in {A, ..., Z}{\expandafter\xdef\csname orcid\x\endcsname{\noexpand\href{https://orcid.org/\csname orcidauthor\x\endcsname}
			{\noexpand\orcidicon}}
}

\author[Reza Mohammadpour]{Reza Mohammadpour\orcidA{} (Uppsala University)\\
\lowercase{reza.mohammadpour@math.uu.se}}

\address{Department of Mathematics, Uppsala University, Box 480, SE-75106, Uppsala, SWEDEN.}
\date{\today}

\subjclass[2010]{28A80, 28D20, 37D35, 37H15 }
\keywords{ Lyapunov exponents, multifractal formalism, topological entropy, typical cocycles}%
\email{reza.mohammadpour@math.uu.se}
\begin{document}
\title[ Entropy spectrum of Lyapunov exponents for typical cocycles ]{ Entropy spectrum of Lyapunov exponents for typical cocycles}

\maketitle
\begin{abstract}
In this paper, we study the size of the level sets of all Lyapunov exponents. For typical cocycles, we establish a variational relation between the topological entropy of the level sets of Lyapunov exponents and the topological pressure of the generalized singular value function.
\end{abstract}

\section{Introduction}
Suppose that $X$ is a compact metric space that is endowed with the
metric $d$. A continuous map $T: X\rightarrow X$ on the compact metric space $X$ is called a topological dynamical system (TDS) and we denote it by $(X,T)$. Let $\mathcal{M}(X)$ be the space of all Borel probability measures on $X$, and $\mathcal{M}(X,T)$ be the space of all $T$-invariant Borel probability measures on $X$.

Assume that $f:X\rightarrow \R$ is a continuous function. Denote by $S_{n}f(x):=\sum_{k=0}^{n-1}f(T^{k}(x))$ the \textit{Birkhoff sum}, and we call 
\[
\lim_{n\rightarrow \infty} \frac{1}{n}S_{n}f(x)
\]
 the \textit{Birkhoff average}.
 
The Birkhoff average converges to the integral of $f$ with respect to the ergodic invariant probability measure $\mu$, almost everywhere. However, there are numerous ergodic invariant measures where the limit exists but converges to a different value. Additionally, there are plenty of points where the Birkhoff average either does not exist or are not considered generic points for any ergodic measure. Therefore, we may ask about the size of the set of points
 \[E_{f}(\alpha)=\bigg\{x\in X : \frac{1}{n} S_{n}f(x)\rightarrow \alpha \hspace{0.2cm}\textrm{as}\hspace{0.2cm}n\rightarrow \infty \bigg\} ,\]
 which we call the $\alpha$\textit{-level set of Birkhoff spectrum}, for a given value $\alpha$ from the set
 \[ L_{f}=\bigg\{\alpha \in \R: \exists x \in X \hspace{0.2cm}\textrm{and} \lim_{n\rightarrow \infty}\frac{1}{n}S_{n}f(x)=\alpha \bigg\},\] 
which we call the \textit{Birkhoff spectrum.}  The size is determined using either the Hausdorff dimension or the topological entropy introduced by
Bowen in \cite{bowen} (See Subsection \ref{top_entropy}).

The topological entropy and Hausdorff spectrum of Birkhoff averages have been intensely
studied by several authors (see, e.g. \cite{BS01, BBS, O}) for different systems and are well understood for continuous potentials.

We say that $\Phi:=\{\log \phi_{n}\}_{n=1}^{\infty}$  is a \textit{subadditive potential} if each $\phi_{n}$ is a continuous positive-valued function on $X$ such that
\[ 0<\phi_{n+m}(x) \leq \phi_{n}(x) \phi_{m}(T^{n}(x)) \quad \forall x\in X, m,n \in \N.\]

Moreover, a sequence of continuous functions (potentials) $\Phi=\{\log\phi_{n}\}_{n=1}^{\infty}$ is said to be an \textit{almost additive potential} if there exists a constant $C > 0$ such that for any $m,n \in \N$, $x\in X$, we have
\[
C^{-1}\phi_{n}(x)\phi_{m}(T^{n})(x) \leq \phi_{n+m}(x)\leq C \phi_{n}(x) \phi_{m}(T^{n}(x)).
\]

By Kingman's subadditive theorem, for any $\mu \in \mathcal{M}(X, T)$ and $\mu$ almost every $x \in X$ such that $\log ^{+}\phi_{1} \in L^{1}(\mu)$, the following limit, called the \textit{top Lyapunov exponent} at $x$, exists:
$$
\chi(x, \Phi):=\lim _{n \rightarrow \infty} \frac{1}{n} \log \phi_{n}(x).
$$

Let $\mathcal{A}: X \rightarrow GL(d, \R)$ be a continuous function over a topological dynamical system $(X,T)$. We denote the product of $\mathcal{A}$ along the orbit of $x$ for time $n$, where $x$ is an element of $X$ and $n$ belongs to the set of natural numbers, as
$$
\mathcal{A}^{n}(x):=\mathcal{A}\left(T^{n-1} (x)\right) \ldots \mathcal{A}(x).
$$
The pair $(\mathcal{A}, T)$ is called a \textit{matrix cocycle}; when the context is clear, we say that $\mathcal{A}$ is a matrix cocycle. That induces a skew-product dynamics $F$ on $X\times \R^{k}$ by $(x, v)\mapsto X\times \R^{k}$, whose $n$-th iterate is therefore \[(x, v)\mapsto (T^{n}(x), \mathcal{A}^{n}(x)v).\] If $T$ is invertible then so is $F$. Moreover, $F^{-n}(x)=(T^{-n}(x), \mathcal{A}^{-n}(x)v)$ for each $n\geq1$, where
\[\mathcal{A}^{-n}(x):=\mathcal{A}(T^{-n}(x))^{-1}\mathcal{A}(T^{-n+1}(x))^{-1}...\mathcal{A}(T^{-1}(x))^{-1}.\]

A well-known example of matrix cocycles is \textit{one-step cocycles} which is defined as follows. Assume that $\Sigma=\{1,...,k\}^{\Z}$ is a symbolic space. Suppose that $T:\Sigma \rightarrow \Sigma$  is a shift map, i.e. $T(x_{l})_{l\in \Z}=(x_{l+1})_{l\in \Z}$. Given a $k$-tuple of matrices $\textbf{A}=(A_{1},\ldots,A_{k})\in \glr^{k}$ , we associate with it the locally constant map $\mathcal{A}:\Sigma \rightarrow \glr$ given by $\mathcal{A}(x)=A_{x_{0}},$ that means the matrix cocycle $\mathcal{A}$ depends only on the zero-th symbol $x_0$ of $(x_{l})_{l\in \Z}$. In this particular situation, we say that $(\mathcal{A}, T)$ is a one-step cocycle; when the context is clear, we say that $\mathcal{A}$ is a one-step cocycle. For any length $n$ word $I=i_{0}, \ldots, i_{n-1},$ (see Section \ref{prel} for the definition) we denote 
\[ \mathcal{A}_{I}:=A_{i_{n-1}}\ldots A_{i_{0}}.\]

 Assume that $\left(A_1, \ldots, A_k\right) \in GL(d, \R)^k$ generates a one-step cocycle $\mathcal{A}:\Sigma \to GL(d, \R).$ Motivated by the study of the multifractal formalism of Birkhoff averages, the level set of the top Lyapunov exponent of certain special subadditive potentials $\Phi=\left\{\log \phi_{n}\right\}_{n=1}^{\infty}$ on full shifts have been studied in \cite{Fe09, FH, BJKR, Moh22}, in which $\phi_{n}(x)=\left\|\mathcal{A}^n(x) \right\|$,  where $\|\cdot\|$ denotes the operator norm. In other words, our focus lies in determining the size of the set of points
 \[E(\alpha)=\bigg\{x\in \Sigma : \frac{1}{n} \log \left\|\mathcal{A}^{n}(x)\right\|\rightarrow \alpha \hspace{0.2cm}\textrm{as}\hspace{0.2cm}n\rightarrow \infty \bigg\} ,\]
 which we call the $\alpha$\textit{-level set of the top Lyapunov exponent}, for a given value $\alpha$ from the set
 \[ L=\bigg\{\alpha \in \R: \exists x \in \Sigma \hspace{0.2cm}\textrm{and} \lim_{n\rightarrow \infty}\frac{1}{n} \log \left\|\mathcal{A}^{n}(x)\right\|=\alpha\bigg\}.\]

The author \cite{Moh22} and Feng \cite{Fe09} calculated the entropy spectrum of the top Lyapunov exponent for generic matrix cocycles.

Let $\mu$ be an ergodic $T$-invariant measure. By Oseledets' theorem, there is a set $Y \subset \Sigma$ of full measure such
that if $x \in Y$ then the Lyapunov exponents $\chi_{1}(x, \mathcal{A}) \geq \chi_{2}(x, \mathcal{A})\geq \ldots \geq \chi_{d}(x, \mathcal{A})$, counted with multiplicity, exist. For $\vec{\alpha}:=(\alpha_1, \ldots, \alpha_d) \in \R^{d}$, we define the $\vec{\alpha}$-level set $I(\vec{\alpha})$ of the Lyapunov exponents by
$$
I(\vec{\alpha})=\bigg\{x \in Y: \chi_{i}(x, \mathcal{A})=\alpha_i \text { for } i=1, \ldots, d \bigg\}
$$
(we emphasize that selecting points in $Y$ implies the assumption of the existence of the limits.). This
suggests to consider, for each $\vec{\alpha}:=(\alpha_1, \ldots, \alpha_d) \in \R^{d}$, the $\vec{\alpha}$-level set
\[E(\vec{\alpha})=\bigg\{ x\in \Sigma: \lim_{n\to \infty}\frac{1}{n}\log \sigma_{i}(\mathcal{A}^{n}(x))=\alpha_i \text{ for }i=1,2, \ldots, d \bigg\},\]
where $\sigma_{1}, \ldots, \sigma_d$ are singular values, listed in decreasing order according to multiplicity. We also define the \textit{Lyapunov spectrum}

\[\vec{L}=\bigg\{\vec{\alpha} \in \R^{d}: \exists x \in \Sigma \text{ such that } \lim_{n \to \infty} \frac{1}{n}\log \sigma_{i}(\mathcal{A}^{n}(x))=\alpha_i \text{ for }i=1,2, \ldots, d \bigg\}.\]

Note that we have $E(\vec{\alpha})=I(\vec{\alpha}) (\bmod 0)$ for every $\vec{\alpha} \in \mathbb{R}^d$. Therefore, one may also think of the $\vec{\alpha}$-level set $E(\vec{\alpha})$ as a level set of the Lyapunov exponents. Our main goal is to calculate the topological entropy of these sets. More precisely, we want to show that the topological entropy of the $\vec{\alpha}$-level set $E(\vec{\alpha})$ is equal to the Legendre transform of the topological pressure for generic matrix cocycles. Note that a similar statement can be expressed in probabilistic language to establish a large deviation principle (LDP) for random matrix products (see \cite{sert}).

 Feng and Huang \cite{FH} calculated the topological entropy of $\vec{\alpha}$-level sets $E(\vec{\alpha})$ for almost additive potentials that improves a result of Barreira and Gelfert in
\cite{BG06} on Lyapunov exponents of nonconformal repellers. Notice that almost additivity condition holds only for a restrictive family of matrices. For instance, Bárány et al. \cite{BKM} showed that this condition for planar matrix
tuples is equivalent to domination, which is an open condition but not generic. One can find more information about the multifractal formalism in \cite{barreira_gelfert, climenhaga, Mohammadpour-survey}.

 In this paper, we consider \textit{typical cocycles} that are matrix cocycles with extra assumptions on some periodic point $p \in \Sigma$ and one of its homoclinic
points $z \in \Sigma$ (see Section \ref{prel} for the precise definition). We call this pair $(p, z)$ a
typical pair, and the typicality assumptions on the pair are suitable generalizations of
the proximality and strong irreducibility in the setting of random product of matrices.
Bonatti and Viana \cite{BV} introduced the notion of typical
cocycles. They showed that the set of typical cocycles is open and dense in
the set of fiber-bunched cocycles and that its complement has infinite codimension. Also, they proved that typical cocycles have simple Lyapunov exponents
with respect to any ergodic measures with continuous local product structure.

We consider typical cocycles  and $\vec{\alpha} \in \R^{d}.$ Then, we calculate the topological entropy of the $\vec{\alpha}$-level set $E(\vec{\alpha})$.   For simplicity, we say the level set $E(\vec{\alpha})$ instead of the $\vec{\alpha}$-level set $E(\vec{\alpha})$ if there is no confusion about $\vec{\alpha}.$

 We define \textit{Falconer's singular value function} $\varphi^{s}(\mathcal{A})$ as follows.  Let $k \in\{0, \ldots, d-1\}$ and $k \leq s<k+1$. Then,
$$
\varphi^{s}(\mathcal{A})=\sigma_{1}(\mathcal{A}) \cdots \sigma_{k}(\mathcal{A}) \sigma_{k+1}(\mathcal{A})^{s-k},
$$

and if $s\geq d,$ then $\varphi^{s}(\mathcal{A})=(\det(\mathcal{A}))^{\frac{s}{d}}.$

 For $q:=(q_{1}, \cdots,  q_{d})\in \R^d$, we define the \textit{generalized singular value function} $\psi^{q_{1}, \ldots, q_{d}}(\mathcal{A}): \mathbb{R}^{d \times d} \rightarrow[0, \infty)$ as
$$
\psi^{q_{1}, \ldots, q_{d}}(\mathcal{A}):=\sigma_{1}(\mathcal{A})^{q_{1}} \cdots \sigma_{d}(\mathcal{A})^{q_{d}}=\left(\prod_{m=1}^{d-1}\left\|\mathcal{A}^{\wedge m}\right\|^{q_{m}-q_{m+1}}\right)\left\|\mathcal{A}^{\wedge d}\right\|^{q_{d}}.
$$

 When $s \in [0, d]$, the singular value function $\varphi^{s}(\mathcal{A}(\cdot))$ coincides with the generalized singular value function $\psi^{q_{1}, \ldots, q_{d}}(\mathcal{A}(\cdot))$ where
$$
\left(q_{1}, \ldots, q_{d}\right)=(\underbrace{1, \ldots, 1}_{m \text { times }}, s-m, 0, \ldots, 0),
$$
with $m=\lfloor s\rfloor$. We denote $\psi^{q}(\mathcal{A}):=\psi^{q_{1}, \ldots, q_{d}}(\mathcal{A}).$ We should notice that even though there is some similarity between the previous expressions, Falconer’s singular value function $\varphi^{s}(\mathcal{A})$ is submultiplicative, whereas the generalized singular value function $\psi^{q}(\mathcal{A})$ is neither submultiplicative nor supermultiplicative.

 Notice that the limit in defining the topological pressure $P \left(\log \psi^{q}(\mathcal{A})\right)$ exists for any $q \in \R^{d}$ when $\mathcal{A}$ is a typical cocycle (see Section \ref{section-upperbound}). 
 \begin{thmA}
 Assume that $\left(A_1, \ldots, A_k\right) \in GL(d, \R)^k$ generates a one-step cocycle $\mathcal{A}:\Sigma \to GL(d, \R).$   Let $\mathcal{A}:\Sigma \to GL(d,\R)$ be a typical cocycle.  Then
  \[h_{\mathrm{top}}(E(\vec{\alpha}))= \inf_{q\in \R^d}\left\{P \left(\log \psi^{q}(\mathcal{A})\right)- \langle q, \vec{\alpha} \rangle \right\}\]
  for all $\vec{\alpha}\in \mathring{\vec{L}}.$ 
  \end{thmA}

In Section 2 we introduce some notation and preliminaries. In Section 3 we prove the upper bound of Theorem A. The key idea in
the proof of Theorem A is to find the dominated subsystems for typical cocycles, and then we prove that the topological pressure over the dominated subset converges to the topological pressure over all points that are done in Section 4. We prove Theorem A for dominated cocycles in Section 5. Finally, we prove Theorem A in Section 6.
  
\subsection{Acknowledgements.} The author thanks Ville Salo, Michal Rams and Kiho Park for helpful discussions. He also thanks the anonymous referee for their valuable corrections and suggestions.

\section{Preliminaries}\label{prel}
\subsection{Subshifts of finite type and standing notations.} 
Assume that $Q=(q_{ij})$ is a matrix $k\times k$ with $q_{ij}\in \{0, 1\}.$  The (two sided) subshift of finite type associated to the matrix $Q$ is a left shift map $T:\Sigma_{Q}\rightarrow \Sigma_{Q}$ i.e., $T(x_{n})_{n\in \Z}=(x_{n+1})_{n\in \Z}$, where $\Sigma_{Q}$ is the set of sequences 
\[\Sigma_{Q}:=\{x=(x_{i})_{i\in \Z} : x_{i}\in \{1,...,k\} \hspace{0.2cm}\textrm{and} \hspace{0.1cm}Q_{x_{i}, x_{i+1}}=1 \hspace{0.2cm}\textrm{for all}\hspace{0.1cm}i\in \Z\};\]
denote it by $(\Sigma_{Q}, T)$. 
In the case where all entries of the matrix $Q$ are equal to 1, we say that is the \textit{full shift}. For simplicity, we denote that  $\Sigma_{Q}=\Sigma$.

 We call that $i_{0}...i_{k-1}$ is an \textit{admissible word} if $Q_{i_{n},i_{n+1}}= 1$ for all $0\leq n \leq k-2$. Let $\mathcal{L}$ denote the collection of admissible words, and $\mathcal{L}_n$ denote the set of admissible words of length $n$. An admissible word of length $n$ is defined as a word $x_0, x_1, ..., x_{n-1}$ such that $x_i \in \{1, 2, ..., k\}$ and $Q_{x_i, x_{i+1}} = 1$. 

In the case of a full shift $(\Sigma, T)$, $\mathcal{L}$ represents the set of all words, while $\mathcal{L}_n$ represents the set of words of length $n$. The length of a word $I$ in $\mathcal{L}$ is denoted by $|I|$. 

We can define the $n$-th level cylinder $[I]$ as follows:
\[ [I]=[i_{0}...i_{n-1}]:=\{x\in \Sigma : x_{j}=i_{j} \hspace{0.2cm}\forall \hspace{0.1cm} 0\leq j\leq n-1 \},\]
for any $i_{0}...i_{n-1}\in \mathcal{L}_n.$

A cylinder containing $x=\left(x_{i}\right)_{i \in \mathbb{Z}} \in \Sigma$ of length $n \in \mathbb{N}$ is defined by
$$
[x]_{n}:=\left\{\left(y_{i}\right)_{i \in \mathbb{Z}} \in \Sigma: x_{i}=y_{i} \text { for all } 0 \leq i \leq n-1\right\}.
$$

We say that the matrix $Q$ is primitive when there exists $n>0$ such that all the entries of $Q^n$ are positive.  The primitivity of $Q$ is equivalent to
the mixing property of the corresponding subshift of finite type $(\Sigma , T)$, and such constant
$n$ is called the mixing rate of $\Sigma.$

We consider the space $\Sigma$ is endowed with the metric $d$ which is  defined as follows: For $x=(x_{i})_{i\in \Z}, y=(y_{i})_{i\in \Z} \in \Sigma$, we have 
\begin{equation}\label{metric}
d(x,y)= 2^{-k},
\end{equation} 
where $k$ is the largest integer such that $x_{i}=y_{i}$ for all $|i| <k.$

In the two-sided dynamics, we define the \textit{local stable set}
\[ W_{\loc}^{s}(x)=\{(y_{n})_{n\in \Z} : x_{n}=y_{n} \hspace{0,2cm}\textrm{for all}\hspace{0.2cm} n\geq 0\} \]
and the \textit{local unstable set}
\[ W_{\loc}^{u}(x)=\{(y_{n})_{n\in \Z} : x_{n}=y_{n} \hspace{0,2cm}\textrm{for all}\hspace{0.2cm} n \leq 0\} .\]

Furthermore, the global stable and unstable manifolds of $x \in \Sigma$ are
\[W^{s}(x):=\left\{y \in \Sigma: T^{n} y \in  W_{\loc}^{s}(T^{n}(x))\text { for some } n \geq 0\right\},\]
\[
W^{u}(x):=\left\{y \in \Sigma: T^{n} y \in W_{\loc}^{u}(T^{n}(x)) \text { for some } n \leq 0\right\}
.\]

The two side subshift of finite type $T:\Sigma \rightarrow \Sigma$ becomes a hyperbolic homeomorphism (see \cite[Subsection 2.3]{AV10}), if $\Sigma$ is equipped by the metric $d$ (see \eqref{metric}).

For any $x, y \in \Sigma$ with $x_{0}=y_{0}$ (i.e., $y \in [x]_{1}$ ), we define
$$
[x, y]:=W_{\loc}^{u}(x) \cap W_{\loc}^{s}(y) .
$$
In particular, $[x, y]$ is the unique point in $[x]_{1}$ whose forward orbit shadows that of $y$ and the backward orbit shadows that of $x$ synchronously.

 \subsection{Multilinear algebra}\label{wedge_product}

Let $A \in \glr.$ We recall that  $\sigma_{1},...,\sigma_{d}$ are the singular values of the matrix $A$. For $A\in GL(d, \R)$, we define an invertible linear map $A^{\wedge l} : \land^{l} \R^{d} \rightarrow \land^{l} \R^{d}$ as follows
\[ (A^{\wedge l}(e_{i_{1}}\wedge e_{i_{2}} \wedge ... \wedge e_{i_{l}}))= Ae_{i_{1}}\wedge Ae_{i_{2}} \wedge ... \wedge Ae_{i_{l}},\]
where $e_i$'s are the standard orthogonal basis of $\R^d.$

$A^{\wedge l}$ can be represented by a $\binom dl \times \binom dl$ whose entries are the $l \times l$ minors of $A$. It can be also shown that 
\[(AB)^{\wedge l}=A^{\wedge l} B^{\wedge l}, \text{and }\|A^{\wedge l}\|=\sigma_{1}(A)...\sigma_{l}(A).\]

\subsection{Fiber bunched}

Let $T:\Sigma \to \Sigma$ be a topologically mixing subshift of finite type. We say that a subadditive potential $\Phi:=\{\log \phi_{n}\}_{n=1}^{\infty}$ over  $(\Sigma, T)$ has \textit{bounded distortion:}  there exists $C\geq 1$ such that for any $n\in \N$ and $I \in \mathcal{L}_{n}$, we have
\[ C^{-1} \leq \frac{\phi_{n}(x)}{\phi_{n}(y)} \leq C\]
for any $x, y \in [I].$

Natural examples of such class of subadditive potentials that has bounded distortion are the singular value potentials $\varphi^{s}(\mathcal{A})$
 of one-step $GL(d, \R)$-cocycles $\mathcal{A}$, where $\varphi^{s}(\mathcal{A}^{n}(x))=\varphi^{s}(\mathcal{A}^{n}(y))$ for any $x, y \in [I].$

 We say that $\mathcal{A}:\Sigma \rightarrow GL(d, \R)$ is an $\alpha$-H\"older continuous function, if there exists $C>0$ such that
\begin{equation}\label{hol}
 \|\mathcal{A}(x)-\mathcal{A}(y)\|\leq Cd(x,y)^{\alpha} \hspace{0,2cm} \forall x,y \in \Sigma.
\end{equation}
For $\alpha> 0$ we let $H^{\alpha}(\Sigma, GL(d, \R))$ be the set of $\alpha$-H\"older continuous functions over the shift with respect to the
metric $d$ on $\Sigma$.

\begin{defn}\label{holonomy}
A \textit{local stable holonomy} for the matrix cocycle $(\mathcal{A}, T)$ is a family of matrices $H_{y \leftarrow x}^{s} \in GL(d, \R)$ defined for all $x\in \Sigma$ with $y\in W_{\loc}^{s}(x)$ such that
\begin{itemize}
\item[a)]$H_{x \leftarrow x}^{s}=Id$ and $H_{z \leftarrow y}^{s} \circ H_{y \leftarrow x}^{s}=H_{z \leftarrow x}^{s}$ for any $z,y \in W_{\loc}^{s}(x)$.
\item[b)] $\mathcal{A}(y)\circ H_{y \leftarrow x}^{s}=H_{T(y) \leftarrow T(x)}^{s}\circ \mathcal{A}(x).$
\item[c)] $(x, y, v)\mapsto H_{y\leftarrow x}(v)$ is continuous.
\end{itemize}
\end{defn}
The local unstable holonomy $H_{y \leftarrow x}^{u}$ is likewise defined as
$$
H_{y \leftarrow x}^{u}:=\lim _{n \rightarrow-\infty} \mathcal{A}^n(y)^{-1} \mathcal{A}^n(x)
$$
for any $y\in W_{\loc}^{u}(x)$, and it satisfies similar properties as those mentioned above, but with $s$ and $T$ replaced by $u$ and $T^{-1}$, respectively.

According to $(b)$ in the above definition, one can extend the definition to the global stable holonomy $H_{y\leftarrow x}^{s}$ for $y\in W^{s}(x)$ not necessarily in $W_{\loc}^{s}(x)$ :
\begin{equation}\label{extension of holonomy}
H_{y\leftarrow x}^{s}=\mathcal{A}^{n}(y)^{-1} \circ H_{T^{n}(y)\leftarrow T^{n}(x)}^{s}\circ \mathcal{A}^{n}(x),
\end{equation} 
where 
$n\in \N$ is large enough such that $T^{n}(y)\in W_{\loc}^{s}(T^{n}(x))$. One can extend the definition of the global unstable holonomy similarly.

\begin{defn}
A $\alpha$-H\"older continuous function $\mathcal{A}$ is called \textit{fiber bunched} if for any $x\in \Sigma$, 
\begin{equation}\label{fiber}
\|\mathcal{A}(x)\|\|\mathcal{A}(x)^{-1}\|\left(\frac{1}{2}\right)^{\alpha}<1.
\end{equation} 

\end{defn}
We say that the matrix cocycle $(\mathcal{A}, T)$ is fiber-bunched if its generator $\mathcal{A}$ satisfies the fiber-bunching condition. We denote by  $H_{b}^{\alpha}(\Sigma, GL(d, \R))$ the space of fiber-bunched cocycles. Under the conditions of H\"older continuity and fiber bunched assumption on $\mathcal{A}\in H_{b}^{\alpha}(\Sigma, GL(d, \R))$, the convergence of the canonical holonomy $H^{s \diagup u}$ is implied. This means that for any $y \in W_{\loc}^{s \diagup u}(x)$, we have the limits
\[H_{y\leftarrow x}^{s} := \lim_{n \rightarrow \infty} \mathcal{A}^{n}(y)^{-1} \mathcal{A}^{n}(x)\]
and
\[H_{y\leftarrow x}^{u} := \lim_{n \rightarrow -\infty} \mathcal{A}^{n}(y)^{-1} \mathcal{A}^{n}(x).\]
Furthermore, the canonical holonomies vary $\alpha$-H\"older continuously (see \cite{KS}), which means that there exists a constant $C > 0$ such that for $y \in W_{\loc}^{s \diagup u}(x)$, we have
\[\|H_{y \leftarrow x}^{s \diagup u} - \mathbb{I}\| \leq C d(x,y)^{\alpha}.\]
It is important to note that in this paper, we will always consider the canonical holonomies. Additionally, it is worth mentioning that the existence of canonical holonomies is guaranteed for one-step cocycles (see \cite[Proposition 1.2]{BV} and \cite[Remark 1]{Moh22}).

Let $x, y \in \Sigma$, $x^{\prime} \in W_{\text {loc }}^{u}(x)$, and $y^{\prime} \in W_{\text {loc }}^{s}(y)$ such that $y^{\prime}=T^{n} (x^{\prime})$ for some $n \in \mathbb{N}$. We can describe these points as forming a path (of length $n$) from $x$ to $y$ via $x'$ and $y'$, which can be represented as:
$$
x \stackrel{W_{\text {loc }}^{u}(x)}{\longrightarrow} x^{\prime} \stackrel{T^{n}}{\longrightarrow} y^{\prime} \stackrel{W_{\text {loc }}^{s}(y)}{\longrightarrow} y.
$$
This path indicates a connection from $x$ to $y$. When considering such a path along with a matrix cocycle $\mathcal{A}: \Sigma \rightarrow GL(d, \R)$ that admits the canonical holonomies, we can define the path cocycle as follows:
$$
B_{x, x^{\prime},y^{\prime},  y}:=H_{y \leftarrow y^{\prime}}^{s} \mathcal{A}^{n}\left(x^{\prime}\right) H_{x^{\prime} \leftarrow x}^{u}.
$$

\subsection{Typical cocycles}
Let $T:\Sigma \to \Sigma$ be a topologically mixing subshift of finite type.
 Suppose that $p\in \Sigma$ is a periodic point of $T$, we say $p\neq z\in \Sigma$ is a \textit{homoclinic point} associated to $p$ if it is the intersection of the stable and unstable manifold of p. That is, $z\in W^{s}(p) \cap W^{u}(p)$. We denote the set of all homoclinic points of $p$ by
$\mathcal{H}(p)$. Then, we define the \textit{holonomy loop} \[W_{p}^{z}:=H_{p \leftarrow z}^s \circ H_{z \leftarrow p}^u. \]
 
Up to replacing $z$ by some backward iterate, we may suppose that $z\in W_{\text{loc}}^{u}(p)$ and $T^{n}(z)\in W_{\text{loc}}^{s}(p)$ for
some $n \geq 1$, which may be taken as a multiple of the period of $p$. Then, by the analogue of
\eqref{extension of holonomy} for stable holonomies,
\[ W_{p}^{z}=\mathcal{A}^{-n}(p)\circ H_{p \leftarrow T^{n}(z)}^{s} \circ \mathcal{A}^{n}(z) \circ H_{z \leftarrow p}^{u}.\]

\begin{defn}\label{typical1}
Suppose that $\mathcal{A}:\Sigma \rightarrow GL(d, \R)$ is a one-step cocycle or that  belongs $H_{b}^{\alpha}(\Sigma, GL(d, \R))$. We say that $\mathcal{A}$ is \textit{1-typical} if there exist a periodic point $p$ and a homoclinic point $z$ associated to $p$ such that:
\begin{itemize}
\item[(i)] The eigenvalues of  $\mathcal{A}^{per(p)}(p)$ have multiplicity $1$ and distinct absolute values,
\item[(ii)] We denote by $\left\{v_{1}, \ldots, v_{d}\right\}$ the eigenvectors of $\mathcal{A}^{per(p)}(p)$, for any $I, J \subset \{1, \ldots, d\}$ with $|I|+$ $|J| \leq d$, the set of vectors
$$
\left\{W_{p}^{z} \left(v_{i}\right): i \in I\right\} \cup\left\{v_{j}: j \in J\right\}
$$
is linearly independent.
\end{itemize}

We say $\mathcal{A}$ is \textit{typical} if $\mathcal{A}^{\wedge t}$ is 1-typical with respect to the same typical pair $(p, z)$ for all $1 \leq t \leq d-1$.
\end{defn}
The fiber-bunching condition in the above definition serves the purpose of ensuring the convergence of the canonical holonomies. However, it is possible to define 1-typicality for matrix cocycles that are not necessarily fiber-bunched, but still have canonical holonomies. For instance, even though the exterior product cocycles $\mathcal{A}^{\wedge t}$ may not be fiber-bunched, they still admit canonical holonomies. Examples of such cocycles can be found among one-step cocycles. Therefore, it is reasonable to consider the 1-typicality assumption on the exterior product cocycles, as we did in Definition \ref{typical1}.

\begin{rem}\label{fixed point} For simplicity, we will always let $p$ be a fixed point by passing to the power $\mathcal{A}^{\text{per}(p)}$
if necessary (because powers of typical cocycles are typical). Moreover, for any homoclinic point $z \in \mathcal{H}(p)$, $T^n (z)$ is a homoclinic point of $p$ for any  $n \in \mathbb{Z}$. That implies that if $z \in \mathcal{H}(p)$ satisfies $(ii)$, then so does any point $T^n (z) \in \mathcal{H}(p)$ in its orbit too. Therefore, we can replace $z$ by any point in its orbit without destroying $(ii)$. 
\end{rem}

The definition of typical cocycles described above is slightly stronger than the typical cocycles introduced by Bonatti and Viana \cite{BV}. In their definition, they only require 1-typicality of $\mathcal{A}^{\wedge t}$ for $1 \leq t \leq d / 2$, and they do not require the typical pair $(p, z)$ to be the same pair over different $t$. Despite our version of typicality being slightly stronger, it is still possible to modify their techniques to prove that the set of typical cocycles is open and dense (for instance, see \cite[Section 5]{BPVL}).

\subsection{Topological pressure}

Let $(X, d)$ be a compact metric space and $T:X\rightarrow X$ be a continuous map. For any $n\in \N$, we define a metric $d_{n}$ on $X$ as follows
\begin{equation}\label{new_metric}
 d_{n}(x, y)=\max\{d(T^{k}(x), T^{k}(y)) : k=0,...,n-1\}.
\end{equation}
For any $\epsilon>0$ a set $E \subset X$ is called to be a $(n,\epsilon)$-\textit{separated  subset}  of $X$ if $d_{n}(x,y)> \epsilon$ (see \eqref{new_metric}) for any two different points $x,y \in E$.

Assume that $\Phi=\{\log \phi_{n}\}_{n=1}^{\infty}$ is a subadditive potential over $(X, T)$. The $\textit{topological pressure}$ of $\Phi$ is defined as follows:
\[ P_{n}(T, \Phi, \epsilon)=\sup \left\{\sum_{x\in E} \phi_{n}(x) : E \hspace{0,1cm}\textrm{is} \hspace{0,1cm}(n, \epsilon) \textrm{-separated subset of }X \right\}.\]
Since $P_{n}(T, \Phi, \epsilon)$ is a decreasing function of $\epsilon$, we define 
 \[P(T, \Phi, \epsilon)=\limsup_{n \rightarrow \infty} \frac{1}{n} \log P_{n}(T, \Phi, \epsilon),\] and
\begin{equation}\label{epsilon}
P(\Phi)=\lim_{\epsilon \rightarrow 0}P(T, \Phi, \epsilon).
\end{equation} 
We call $ P(\Phi)$ the \textit{topological pressure} of $\Phi$, whose existence of the limit is guaranteed from the subadditivity of $\Phi.$

\begin{rem}\label{(n, 1)separated sets}Since the subshift of finite type $\left(\Sigma, T\right)$ is expansive \footnotemark\footnotetext{1 is an expansive constant.}, the pressure $P(\Phi)$ may be expressed as follows:
\[P(\Phi)=\limsup _{n \rightarrow \infty} \frac{1}{n} \log \sup \left\{\sum_{x \in E} \phi_{n}(x):E \hspace{0,1cm}\textrm{is} \hspace{0,1cm}(n, 1) \textrm{-separated subset of }\Sigma \right\};\]
that is, we can compute the pressure by looking at $(n, 1)$-separated sets, and drop the limit in $\epsilon$ from the definition of the pressure \eqref{epsilon}.
\end{rem}

Cao, Feng, and Huang \cite{CFH} prove the subadditive variational principle:
\begin{equation}\label{varitional}
P(\Phi)=\sup \bigg\{h_{\mu}(T)+\chi(\mu, \Phi): \mu \in \mathcal{M}(X, T) \bigg\},
\end{equation}
where $h_{\mu}(T)$ is the measure-theoretic entropy and \[\chi(\mu, \Phi):=\lim_{n\to \infty}\int \frac{\log \phi_{n}(x)}{n}d\mu(x).\]

An invariant measure $\mu \in \mathcal{M}(X, T)$ achieving the
supremum in \eqref{varitional} is called an \textit{equilibrium state} of $\Phi$. Moreover, at least one equilibrium
state necessarily exists for any subadditive potential $\Phi$ if the entropy map $\mu \mapsto h_{\mu}(T)$
is upper semi-continuous.

\subsection{Topological entropy}\label{top_entropy}

Let $T:X \to X$ be a topological dynamical system. For any $n\in \N$ and $\epsilon>0$ we define \textit{Bowen ball} $B_{n}(x, \epsilon)$ as follows:
\[ B_{n}(x, \epsilon)=\{y\in X : d_{n}(x, y)<\epsilon\}.\]
Consider a set $Y \subset X$ and $\epsilon>0$. We define a covering of $Y$ as a countable collection of balls $\mathcal{Y}:=\left\{B_{n_{i}}\left(y_{i}, \epsilon\right)\right\}_{i}$ such that $Y$ is contained within the union of these balls, i.e., $Y \subset \bigcup_{i} B_{n_{i}}\left(y_{i}, \epsilon\right)$. For a given collection $\mathcal{Y}=\left\{B_{n_{i}}\left(y_{i}, \epsilon\right)\right\}_{i}$, we define $n(\mathcal{Y})$ as the minimum value of $n_i$ among all indices $i$.  Let $s\geq 0$ and define
\[ S(Y, s, N, \epsilon)=\inf \sum_{i} e^{-sn_{i}},\]
 where the infimum is taken over all collections $\mathcal{Y}=\{B_{n_{i}}(x_{i}, \epsilon)\}_{i}$ that cover $Y$ and satisfy $n(\mathcal{Y})\geq N$. As $S(Y, s, N, \epsilon)$ is non-decreasing with respect to $N$, the limit  $S(Y, s, N, \epsilon)$ exists
 \[ S(Y, s, \epsilon): = \lim_{N\rightarrow \infty} S(Y, s, N, \epsilon).\]
 There is a critical value of the parameter $s$, which we denote by $h_{top}(T,Y, \epsilon)$ such that
\[S(Y, s, \epsilon)=\left\{\begin{array}{ll}
         0, & \mbox{$s>h_{top}(T,Y, \epsilon)$},\\
        \infty, & \mbox{$s<h_{top}(T,Y, \epsilon)$}.\end{array} \right . \] 
        
        Since $h_{top}(T,Y, \epsilon)$ does not decrease with $\epsilon$, the following limit exists,
        \[h_{top}(T,Y)=\lim_{\epsilon \rightarrow 0}(T,Y,\epsilon).\]
        We call $h_{top}(T, Y)$ the \textit{topological  entropy}  of $T$ restricted  to $Y$ or the topological entropy  of $Y$ (we denote $h_{top}(Y)$), as  there  is  no  confusion  about $T$. We denote $h_{top}(X, T)=h_{top}(T).$  For subshifts of finite type with the $d$ metric, we can just take $\epsilon = 1.$
\section{Upper bound}\label{section-upperbound}
 Assume that $\left(A_1, \ldots, A_k\right) \in GL(d, \R)^k$ generates a one-step cocycle $\mathcal{A}:\Sigma \to GL(d, \R).$

We say that a one-step cocycle $\mathcal{A}: \Sigma \to GL(d, \R)$ is \textit{simultaneously quasi-multiplicative} if there exist $C>0$ and $k \in \N$ such that for all $I, J \in \mathcal{L}$, there is $K=K(I, J) \in \mathcal{L}_{k}$ such that $IKJ \in \mathcal{L}$ and for each $i \in \{1, \ldots, d-1\}$, we have 
\[ \|\mathcal{A}_{IKJ}^{\wedge i}\|\geq C \|\mathcal{A}_{I}^{\wedge i}\| \|\mathcal{A}^{\wedge i}_{J}\|.\]

For any $q \in \R^d$, note that $\psi^q(\mathcal{A})$ is neither submultiplicative nor supermultiplicative. For one-step cocycles, the limsup topological pressure of $\log \psi^q(\mathcal{A})$ can be defined by
\[ P^*(\log \psi^q(\mathcal{A})):=\limsup_{n \to \infty}\frac{1}{n} \log s_n(q),     \hspace{0.5cm} \forall q \in\R^d, \]
where $s_{n}(q):=\sum_{I \in \mathcal{L}_n} \psi^q(\mathcal{A}_{I})$.
When the limit exists, we denote the topological pressure by $P(\log \psi^q(\mathcal{A})).$ In the following Lemma, we prove that the limit exists under the certain assumption.

\begin{lem}\label{topological_pressure}
Assume that a one-step cocycle $\mathcal{A}:\Sigma \to GL(d, \R)$ is simultaneously quasi-multiplicative. Then the limit in defining  $P^*(\log \psi^q(\mathcal{A}))$ exists for any $q\in \R^{d}.$
\end{lem}
\begin{proof}

By the simultaneous quasi-multiplicativity property of $\mathcal{A}$, there exist $C>0$ and $k \in \N$ such that for all  $m, n > k$,  $I \in \mathcal{L}_{n}$ and $J \in \mathcal{L}_{m}$, there is $K \in \mathcal{L}_{k}$ such that $IKJ \in \mathcal{L}$ and for each $i \in \{1, \ldots, d-1\}$, we have 
\[ \|\mathcal{A}^{\wedge i}_{IKJ}\|\geq C \|\mathcal{A}^{\wedge i}_{I}\| \|\mathcal{A}^{\wedge i}_{J}\|.\]
For any $q=(q_1, \ldots, q_d)\in \R^{d}$, we can write
\[\psi^q(\mathcal{A}_{IKJ})=\underbrace{\prod_{i=1}^{d} \|\mathcal{A}^{\wedge i}_{IKJ}\|^{t_{i}}}_\text{(1)},\]
where $t_{i}=q_{i}-q_{i+1},$ and $q_{d+1}=0$ for $1\leq i \leq d.$

If $t_{i}< 0$, then by the sub-multiplicativity property, there is $C_{0}>0$ such that
\begin{equation}\label{submultiplicative}
 \|\mathcal{A}^{\wedge i}_{IKJ}\|^{t_{i}} \geq C_{0}^{t_{i}} \|\mathcal{A}^{\wedge i}_{I}\|^{t_{i}} \|\mathcal{A}^{\wedge i}_{J}\|^{t_{i}}.
\end{equation}

If $t_{i}\geq  0$, then by the simultaneous quasi-multiplicativity of $\mathcal{A}$, we have
\begin{equation}\label{quasi-mult_eq}
 \|\mathcal{A}^{\wedge i}_{IKJ}\|^{t_{i}} \geq C^{t_{i}} \|\mathcal{A}^{\wedge i}_{I}\|^{t_{i}} \|\mathcal{A}^{\wedge i}_{J}\|^{t_{i}}.
\end{equation}

By \eqref{submultiplicative} and \eqref{quasi-mult_eq},
\[ (1) \geq C_{1} \prod_{i=1}^{d} \|\mathcal{A}^{\wedge i}_{I}\|^{t_{i}} \prod_{i=1}^{d} \|\mathcal{A}^{\wedge i}_{J}\|^{t_{i}},\]
where $C_{1}:=C_{1}(C_{0}^{t_{i}}, C^{t_{i}}).$ Therefore,
\[\psi^q(\mathcal{A}_{IKJ}) \geq C_1 \psi^q(\mathcal{A}_{I})\psi^q(\mathcal{A}_{J}).\]

Then,
\[s_{n+k+m}(q)\geq C_{1} s_{n}(q)s_{m}(q).\]

We denote $a_{n}:=\frac{ s_{n-k}(q)}{C_{1}}.$ Thus,  $a_{n+m} \geq a_{n}a_{m}$ for all $m, n>k$ . Hence, the limit exists.

\end{proof}
Let $\mathcal{A}:\Sigma \to GL(d, \R)$ be a one-step cocycle. We define 

$$
 \Phi_{\mathcal{A}}:=\left( \log \sigma_{1}(\mathcal{A}), \ldots, \log\sigma_{d}(\mathcal{A})\right).
$$
 For any $q\in \R^d$,  $\langle q, \Phi_{\mathcal{A}} \rangle =\log \psi^q(\mathcal{A}).$ We denote 
 \[P_{n}(T, \langle q, \Phi_{\mathcal{A}} \rangle, 1):=\sup \left\{\sum_{x\in E} \psi^q(\mathcal{A}^{n}(x)) : E \hspace{0,1cm}\textrm{is} \hspace{0,1cm}(n, 1) \textrm{-separated subset of }\Sigma \right\}.\]
 Then, by Lemma \ref{topological_pressure}, $P\left(\log \psi^q(\mathcal{A})\right)=P(\langle q, \Phi_{\mathcal{A}} \rangle)=\lim_{n\to \infty} \frac{1}{n}\log P_{n}(T, \langle q, \Phi_{\mathcal{A}} \rangle, 1).$
\begin{thm}\label{HD-one side}Assume that a one-step cocycle $\mathcal{A}:\Sigma \to GL(d, \R)$   is simultaneously quasi-multiplicative.  Then, 
\[h_{\mathrm{top}}\left( E(\vec{\alpha})\right) \leq \inf _{t \in \mathbb{R}^{d}}\left\{P\left(\log \psi^t(\mathcal{A})\right)- \langle t, \vec{\alpha} \rangle \right\}
\]
for all $\alpha \in \mathring{\vec{L}}$.

\end{thm}
\begin{proof}

 For any $\vec{\alpha}=(\alpha_1, \ldots, \alpha_d) \in \mathring{\vec{L}}$ and $r>0$, we define
\[G(\vec{\alpha}, n, r):=\bigg\{x \in \Sigma:\hspace{0.1cm}  \bigg|\frac{1}{m}  \log \sigma_{i}(\mathcal{A}^{m}(x))-\alpha_i \bigg|<\frac{1}{r} \text{  for all }1 \leq i \leq d \text{ and } m \geq n \bigg\}.\]

It is clear that for any $r>0$,
\[ E(\vec{\alpha}) \subset \bigcup_{n=1}^{\infty}  G(\vec{\alpha}, n, r ).\]

Note that the limit in defining $P\left(\log \psi^q(\mathcal{A})\right)$ exists by Lemma \ref{topological_pressure}.

\textit{Claim.} For any $q=\left(q_{1}, \cdots, q_{d}\right) \in \mathbb{R}^d$,
$$
h_{\mathrm{top}}( G(\vec{\alpha}, n, r )) \leq P (\log \psi^q(\mathcal{A}))-\sum_{i=1}^{d}\left(\alpha_i q_i-\frac{|q_i|}{r}\right).
$$

\textit{Proof of the claim.} Let $s<h_{\mathrm{top}}(G(\vec{\alpha}, n, r ))$ be given. By definition (see Subsection \ref{top_entropy}),  $h_{\mathrm{top}}( G(\vec{\alpha}, n, r ), 1)>s$. Thus,
$$
\infty=S(G(\vec{\alpha}, n, r ), s, 1)=\lim _{N \rightarrow \infty} S(G(\vec{\alpha}, n, r ), s, N, 1) .
$$
Hence there exists $N_{0}$ such that
$$
S(G(\vec{\alpha}, n, r ), s, N, 1) \geq 1, \quad \forall N \geq N_{0} .
$$
Now take $N \geq \max \left\{n, N_{0}\right\}$ and let $E$ be a $(N, 1)$-separated subset of $G(\vec{\alpha}, n, r )$ with the maximal cardinality. Then $\bigcup_{x \in E} B_{N}(x, 1) \supseteq G(\vec{\alpha}, n, r )$. It follows
\begin{equation}\label{seprated_set}
\# E \cdot \exp (-s N) \geq S(G(\vec{\alpha}, n, r ), s, N, 1) \geq 1 .
\end{equation}
Since $$\sum_{i=1}^{d} q_{i} \log \sigma_{i}\left(\mathcal{A}^{N}(x)\right) \geq N \left(\sum_{i=1}^{d}\left(\alpha_{i} q_{i}-\frac{\left|q_{i}\right|}{r}\right)\right)$$ for each $x \in G(\vec{\alpha}, n, r )$ and $q\in \R^{d}$, we have 
$$
\begin{aligned}
P_{N}(T, \langle q, \Phi_{\mathcal{A}} \rangle, 1)& \geq \sum_{x \in E} \exp \left(\sum_{i=1}^{d} q_{i} \log \sigma_{i}(\mathcal{A}^{N}(x))\right)\\
& \geq \# E \cdot \exp \left(N\left(\sum_{i=1}^{d} \left(\alpha_i q_i -\frac{|q_i|}{r}\right) \right) \right).
\end{aligned}
$$
By \eqref{seprated_set}, \[P_{N}(T, \langle q, \Phi_{\mathcal{A}} \rangle, 1) \geq \exp \left(N\left(s+\sum_{i=1}^{d}\left(\alpha_i q_i -\frac{|q_i|}{r}\right)\right) \right).\]
 Letting $N \rightarrow \infty$, we obtain $P( \langle q, \Phi_{\mathcal{A}} \rangle) \geq s+ \sum_{i=1}^{d}\left(\alpha_i q_i-\frac{|q_i|}{r}\right)$ by Lemma \ref{topological_pressure}. Hence we have
$$
P(\log \psi^q(\mathcal{A})) \geq s+\sum_{i=1}^{d}\left(\alpha_i q_i-\frac{|q_i|}{r}\right),
$$
which finished the proof of the claim.

Thus,
$$
\begin{aligned}
 h_{\mathrm{top}}(E(\vec{\alpha}))
 &\leq  h_{\mathrm{top}}\bigg(\bigcup_{n=1}^{\infty}G(\vec{\alpha}, n, r )\bigg)\\
 & \leq \sup_{n \geq 1}h_{\mathrm{top}} (G(\vec{\alpha}, n, r ))\\
 &\leq P (\log \psi^q(\mathcal{A}))-\sum_{i=1}^{d}\left(\alpha_i q_i-\frac{|q_i|}{r}\right).
\end{aligned}
$$

By
letting $r \to \infty$,
\[
h_{\mathrm{top}}(E(\vec{\alpha})) \leq P\left(\log \psi^q(\mathcal{A})\right)- \langle q,  \vec{\alpha} \rangle .
\]

Thus, \[
h_{\mathrm{top}}(E(\vec{\alpha})) \leq \inf _{t \in \mathbb{R}^{d}}\left\{P\left(\log \psi^t(\mathcal{A})\right)- \langle t,  \vec{\alpha} \rangle \right\}.
\]

\end{proof}

The upper bound of Theorem A follows from the below corollary.

\begin{cor}\label{proof-upper-bound}
 Assume that $\left(A_1, \ldots, A_k\right) \in GL(d, \R)^k$ generates a one-step cocycle $\mathcal{A}:\Sigma \to GL(d, \R).$ Let $\mathcal{A}:\Sigma \to GL(d, \R)$ be a typical cocycle. Then,
\[h_{\mathrm{top}}\left( E(\vec{\alpha})\right) \leq  \inf _{q \in \mathbb{R}^{d}}\left\{P\left(\log \psi^q(\mathcal{A})\right)- \langle q, \vec{\alpha} \rangle  \right\}
\]
for all $\alpha \in \mathring{\vec{L}}$.
\end{cor}
\begin{proof}
By \cite[Theorem 4.1]{Park20}, typical cocycles are simultaneously quasi-multiplicative. Therefore, the proof follows from Theorem \ref{HD-one side}.
\end{proof}

\begin{rem}
Let us stress that the simultaneous quasi-multiplicativity property that was proved in \cite[Theorem 4.1]{Park20} is weaker than our simultaneous quasi-multiplicativity; the connecting word $K$ does not have a fixed length in \cite{Park20}. But we found out that Park \cite[Theorem 4.1]{Park20} actually proved the simultaneous quasi-multiplicativity property, which we consider.  We explain how to modify parts of the proof of \cite[Theorem 4.1]{Park20} to obtain our simultaneous quasi-multiplicativity: 

For simplicity, we used the same notations as \cite{Park20}. $K$ is introduced prior to the statement of Lemma 4.19 in \cite{Park20}. There is no control on $n \in \{0,1, \ldots, N\}$ in \cite[Lemma 4.22]{Park20} which is why the length of $K$ is not fixed.  But we can choose  $\ell$ bigger than $\ell_0$ in \cite[Lemma 4.13]{Park20} such that $\ell+n$ is uniform. Therefore, the length of $K$ is fixed.

\end{rem}
 Bárány and Troscheit \cite[Proposition 2.5]{BT22} proved that if a one-step cocycle $\A$ is fully strongly irreducible and proximal, then $\A$ is simultaneously quasi-multiplicativite. Recently, the author and Park \cite{MP-uniform-qm} generalized their result for the norm of matrix cocycles under the ireducibility assumption (see \cite[Corollary 1.2]{MP-uniform-qm}).

  \section{Dominated subsystems}\label{dominated_subsystem}
  Suppose that $T:X\rightarrow X$ is a diffeomorphism on a compact invariant set $X$. Let $V\oplus W$ be a splitting of the tangent bundle over $X$ that is invariant by the tangent map $DT$. In this case, if vectors in $V$ are uniformly contracted by $DT$ and vectors in $W$ are uniformly expanded, then this splitting is called \textit{hyperbolic}. The more general notion is the \textit{dominated splitting}, if at each point all vectors in $W$ are more expanded than all vectors in $V$. Domination is  equivalent to $V$ being hyperbolic repeller and $W$ being hyperbolic attractor in the projective bundle.

 In the matrix cocycles, we are interested in bundles of the form $X\times \R^{d}, $ where the matrix cocycles are generated by $(\mathcal{A}, T)$. It was showed in \cite{BGO} that a cocycle admits a dominated splitting $V\oplus W$ with $\dim W=i$ if and only if when $n\rightarrow \infty$, the ratio between the $i$-th and $(i+ 1)$-th singular values of the matrices of the $n$-th iterate increase uniformly exponentially. 
  \begin{defn}
We say that a matrix cocycle $\mathcal{A}$ is \textit{dominated with index $i$} if there exist constants $C >1$ and $0<\tau<1$ such that
\[\frac{\sigma_{i+1}(\mathcal{A}^{n}(x))}{\sigma_{i}(\mathcal{A}^{n}(x))}< C \tau^n, \hspace{0.2cm} \forall n\in \N, x\in X.\]

We say that the cocycle $\mathcal{A}$ is \textit{dominated} if $\mathcal{A}$ is dominated with index $i$ for all $i\in \{1, \ldots, d-1\}.$
\end{defn}

Let $\textbf{A}$ be a compact set in $GL(d, \R).$ We say that $\textbf{A}$ is \textit{dominated} with index  $i$ iff there exist $C>0$  and $0<\tau<1$ such that for any finite sequence $A_1, \ldots, A_N$ in $\textbf{A}$ we have
$$
\frac{\sigma_{i+1}\left(A_1 \cdots A_N\right)}{\sigma_i\left(A_1 \cdots A_N\right)}<C \tau^N.
$$
We say that $\textbf{A}$ is dominated iff it is dominated with index  $i$ for each $i\in\{1, \ldots, d-1\}.$ A one step cocycle $\A$ generated by $\textbf{A}$ is dominated if $\textbf{A}$ is dominated.

 \begin{rem}\label{dominated-wedge}
 According to the multilinear algebra properties, $\mathcal{A}$ is dominated with index $k$ if and only if $\mathcal{A}^{\wedge k}$ is dominated with index 1. Therefore, the cocycle $\mathcal{A}$ is dominated if and only if $\mathcal{A}^{\wedge k}$ is dominated with index $k$ for any $k \in \{1, \ldots, d-1\}.$
 \end{rem}

Let $\mathbb{V}$ be a finite
dimensional $\R$-vector space equipped with a norm $\| \cdot \|$ and denote by $\mathbb{P}(\mathbb{V})$ its projective space.  For $v \in \mathbb{P}(\mathbb{V})$, let the \textit{cone} around $v$ of size $\varepsilon$ be defined as
$$
\mathcal{C}(v, \varepsilon):=\{w \in \mathbb{P}(\mathbb{V}): \angle(v, w)<\varepsilon\} .
$$

 Let  $\mathring{\mathcal{C}}(v_{1}, \varepsilon)$ denote the interior of the cone $\mathcal{C}(v, \varepsilon)$. Arguing as in \cite[Theorem 3.1]{Par22}, we can prove the following result.
 \begin{thm}\label{dominated-th}
 Let $\mathcal{A}:\Sigma \to GL(d,\R)$ be a 1-typical cocycle with distinguished fixed point $p$ (see Remark \ref{fixed point}). Then there exist $\varepsilon>0$ and $K \in \N$  such that for any $x\in \Sigma$ and $n\in \N$, there exists $\omega:=\omega_{x}\in \Sigma$ such that 
 \begin{itemize}
 \item[1)]there is $m_1:=m_1(\omega) \in \N$ such that $T^{m_1}(\omega)\in [x]_n$, and
 \item[2)]there is $m_2:=m_2(\omega)\in \N$ such that a path cocycle \begin{equation}\label{path-main}
  B_{p, \omega, T^{m_1+n+m_2}(\omega), p}= H_{p \longleftarrow T^{m_1+n+m_2}(\omega)}^{s}\mathcal{A}^{m_1+n+m_2}(\omega)H_{\omega \longleftarrow p}^u
 \end{equation}

 from $p$ to $p$ of length $m_1+n+m_2$ satisfying \[
  B_{p, \omega, T^{m_1+n+m_2}(\omega), p}\mathcal{C}(v_{1}, \varepsilon)\subset \mathring{\mathcal{C}}(v_{1}, \varepsilon),
\]
  
 \end{itemize}
where $v_1$ is the eigendirection $\mathcal{A}(p)$ corresponding to the eigenvalue of the largest modulus and $m_i\leq K$ for $i=1, 2.$
 \end{thm}
 \begin{proof}
 This follows from \cite[Lemma 3.10]{Par22}\footnotemark\footnotetext{\cite[Lemma 3.10]{Par22} is the last part of the proof of \cite[Theorem 3.1]{Par22}, where you can see the corresponding path  $\widetilde{B}$ has the desired property listed in \cite[Theorem 3.1]{Par22}.}. 
 \end{proof}

We recall that the canonical holonomies always exist for one-step cocycles (see \cite[Proposition 1.2]{BV} and \cite[Remark 1]{Moh22}).

\begin{cor}\label{dominated_one_step}
 Assume that $\left(A_1, \ldots, A_k\right) \in GL(d, \R)^k$ generates a one-step cocycle $\mathcal{A}:\Sigma \to GL(d, \R).$ Suppose that $\mathcal{A}:\Sigma \to GL(d,\R)$ is a 1-typical cocycle.  Then, there exists $K \in \N$ such that for every $n\in \N$ and $I \in \mathcal{L}_{n}$ there exist $J_2:=J_2(I)$ and $J_1:=J_1(I)$ with $|J_{i}|\leq K$ for $i=1,2$ such that the tuple\[\left(\mathcal{A}_{\mathrm{k}}\right)_{\mathrm{k} \in \mathcal{L}_{\ell(n)}^{D}}, \quad \text{where } \mathcal{L}_{\ell(n)}^{D}:=\left\{J_1(I) IJ_2(I): I \in \mathcal{L}_{n}\right\},\]
is dominated with index 1.
\end{cor}

\begin{proof}
We consider the path \eqref{path-main} in Theorem \ref{dominated-th}. By using the properties of the canonical holonomies,$$\begin{aligned}
& H_{p \longleftarrow T^{m_1+n+m_2}(\omega)}^{s}\mathcal{A}^{m_1+n+m_2}(\omega)H_{\omega \longleftarrow p}^u=\\
& \underbrace{H_{p \longleftarrow T^{m_2}(T^{n+m_1}(\omega))}^{s}\mathcal{A}^{m_2}(T^{n+m_1}(\omega))  H_{T^{m_1+n}(\omega)\longleftarrow T^{n}(z_1)}^{u}}_\textrm{(2)}\\
&H_{T^{n}(z_1)\longleftarrow T^{m_1+n}(\omega)}^{u}  H_{T^{m_1+n}(\omega)\longleftarrow T^{n}(x)}^{s}\mathcal{A}^{n}(x)\\
&\underbrace{H_{x \longleftarrow T^{m_{1}}(\omega)}^s \mathcal{A}^{m_{1}}(\omega) H_{\omega \longleftarrow p}^u}_\textrm{(1)},
  \end{aligned}
  $$
  where $z_1:= [T^{m_{1}}(\omega), x].$

We consider $I$ as length $n$ orbit of $x$, $J_2:=J_2(I)$ as the path (2) of length $m_2$ from $T^{n}(z_1)$ to $p$ and $J_1:=J_1(I)$ as the path (1) of length $m_1$ from $p$ to $x$. By Theorem \ref{dominated-th}, there is $K \in \N$ such that $|J_1|+|J_2| \leq 2K.$  Taking into account that $H_{T^{n}(z_1)\longleftarrow T^{m_1+n}(\omega)}^{u},  H_{T^{m_1+n}(\omega)\longleftarrow T^{n}(x)}^{s}$ are identity, the path $\mathcal{A}_{J_2}\mathcal{A}_{I}\mathcal{A}_{J_1}$ \footnotemark\footnotetext{$\mathcal{A}_{J_2}\mathcal{A}_{I}\mathcal{A}_{J_1}=\mathcal{A}_{J_1IJ_2}$.}   maps $\mathcal{C}(v_{1}, \tau)$ into its interior by Theorem \ref{dominated-th}.  Bochi and Gourmelon \cite[Theorem B]{BGO} showed that a one-step cocycle is dominated with index 1 if and only if there exists a strongly invariant multicone  that can be checked for finite set or finite sequences. Therefore,  we obtain the assertion. 
\end{proof}

Let $\mathcal{A}$ be a typical cocycle. For $t\in \{1, \ldots, d-1\}$,  we denote by $H^{s / u, t}:=\left(H^{s / u}\right)^{\wedge t}$ the canonical holonomies of $\mathcal{A}_{t}$, where $H^{s / u}$ are the canonical holonomies of $\mathcal{A}$.  We also set $\mathbb{V}_{t}:=\mathbb{R}^{\left( \begin{array}{l}d \\ t\end{array}\right)}$ and $\mathcal{A}_t:=\mathcal{A}^{\wedge t}$  for $t\in \{1, \ldots, d-1 \}$.

 Arguing as in \cite[Theorem 4.6]{Par22}, we can prove the following result.
\begin{thm}\label{dominated_typical}
Let $\mathcal{A}_t:\Sigma \to GL(\mathbb{V}_t)$ be a 1-typical cocycle with distinguished fixed point $p$ 
and its homoclinic point $z$ for all $t \in \{1, \ldots, d-1\}$ (see Remark \ref{fixed point}). Then there exist $\varepsilon>0$ and $K_{0} \in \N$  such that for any $x\in \Sigma$ and $n\in \N$, there exists $\omega:=\omega_{x}\in \Sigma$ such that
 \begin{itemize}
 \item[1)]there is $m_{1}:=m_1(\omega) \in \N$ such that $T^{m_1}(\omega)\in [x]_n,$ and
 \item[2)]there is $m_2:=m_2(\omega)\in \N$ such that a path cocycle \begin{equation}\label{path-main1}
  B_{p, \omega, T^{m_1+n+m_2}(\omega), p}^{t}:= H_{p \longleftarrow T^{m_1+n+m_2}(\omega)}^{s, t}\mathcal{A}_t^{m_1+n+m_2}(\omega)H_{\omega \longleftarrow p}^{u, t}
 \end{equation}

 from $p$ to $p$ of length $m_1+m_2+n$ satisfying \[
  B_{p, \omega, T^{m_1+n+m_2}(\omega), p}^t\mathcal{C}(v_{1}^{t}, \varepsilon)\subset \mathring{\mathcal{C}}(v_{1}^{t}, \varepsilon),
\]
 \end{itemize}
for all $t \in\{1, \ldots, d-1\}$, where $v_1^{t}$ is the eigendirection $\mathcal{A}_t (p)$ corresponding to the eigenvalue of the largest modulus and $m_i \leq K_{0}$ for $i=1,2.$
 
 \end{thm}

   Note that by multilinear algebra properties,
  \begin{equation}\label{wedge-path}
  B_{p, \omega, T^{m_1+n+m_2}(\omega), p}^{t}=B_{p, \omega, T^{m_1+n+m_2}(\omega), p}^{\wedge t}.
  \end{equation}

 \begin{cor}\label{dominated-one-step-cocycle}
 Assume that $\left(A_1, \ldots, A_k\right) \in GL(d, \R)^k$ generates a one-step cocycle $\mathcal{A}:\Sigma \to GL(d, \R).$ Suppose that $\mathcal{A}:\Sigma \to GL(d,\R)$ is a typical cocycle. Then, there exists $K_0 \in \N$ such that for every $n\in \N$ and $I \in \mathcal{L}_{n}$ there exist $J_2=J_2(I)$ and $J_1=J_1(I)$ with $|J_{i}|\leq K_0$ for $i=1,2$ such that the tuple \[\left(\mathcal{A}_{\mathrm{k}}\right)_{\mathrm{k} \in \mathcal{L}_{\ell(n)}^{\mathcal{D}}}, \quad \text{where } \mathcal{L}_{\ell(n)}^{\mathcal{D}}:=\left\{J_1(I) IJ_2(I): I \in \mathcal{L}_{n}\right\},\]
is dominated.
\end{cor} 
 \begin{proof}
We consider the path \eqref{path-main1} in Theorem \ref{dominated_typical}. By using the properties of the canonical holonomies, $$\begin{aligned}
& B_{p, \omega, T^{m_1+n+m_2}(\omega), p}^{t}=H_{p \longleftarrow T^{m_1+n+m_2}(\omega)}^{s, t}\mathcal{A}_{t}^{m_1+n+m_2}(\omega)H_{\omega \longleftarrow p}^{u, t}= \\
& \underbrace{H_{p \longleftarrow T^{m_2}(T^{n+m_1}(\omega))}^{s, t}\mathcal{A}_{t}^{m_2}(T^{n+m_1}(\omega))  H_{T^{m_1+n}(\omega)\longleftarrow T^{n}(z_2)}^{u, t}}_\textrm{(2)}\\
&H_{T^{n}(z_2)\longleftarrow T^{m_1+n}(\omega)}^{u, t}  H_{T^{m_1+n}(\omega)\longleftarrow T^{n}(x)}^{s, t}\mathcal{A}_{t}^{n}(x)\\
&\underbrace{H_{x \longleftarrow T^{m_{1}}(\omega)}^{s, t} \mathcal{A}_{t}^{m_{1}}(\omega) H_{\omega \longleftarrow p}^{u, t}}_\textrm{(1)},
  \end{aligned}
  $$
  where $z_2:= [T^{m_{1}}(\omega), x].$

  We consider $I$ as length $n$ orbit of $x$, $J_2:=J_2(I)$ as the path (2) of length $m_2$ from $T^{n}(z_2)$ to $p$ and $J_1:=J_1(I)$ as the path (1) of length $m_1$ from $p$ to $x$. By Theorem \ref{dominated_typical}, there is $K_0 \in \N$ such that $|J_1|+|J_2| \leq 2K_0.$  Taking into account that \(H_{T^{n}(z_2)\longleftarrow T^{m_1+n}(\omega)}^{u, t}\) and \(H_{T^{m_1+n}(\omega)\longleftarrow T^{n}(x)}^{s, t}\) are identity, the path $\mathcal{A}^{\wedge t}_{J_2}\mathcal{A}^{\wedge t}_{I}\mathcal{A}^{\wedge t}_{J_1}$ \footnotemark\footnotetext{$\mathcal{A}^{\wedge t}_{J_2}\mathcal{A}^{\wedge t}_{I}\mathcal{A}^{\wedge t}_{J_1}=\mathcal{A}^{\wedge t}_{J_1IJ_2}$.} maps \(\mathcal{C}(v_{1}^{t}, \tau)\) into its interior by Theorem \ref{dominated_typical} for any \(t \in \{1, \ldots, d-1\}\).  Bochi and Gourmelon \cite[Theorem B]{BGO} showed that a one-step cocycle $\mathcal{A}^{\wedge t}$ is dominated with index 1 for each $t\in \{1, \ldots, d-1 \}$ if and only if there exists a strongly invariant multicone  for $\mathcal{A}^{\wedge t}$ for each $t\in \{1, \ldots, d-1 \}$ that can be checked for finite set or finite sequences. Therefore,  we obtain the assertion (see Remark \ref{dominated-wedge}). 
 \end{proof}
 The corollary \ref{dominated-one-step-cocycle} extends a similar result in \cite{BJKR} that was proved in the two-dimensional setting.

For simplicity, we denote by $\ell:=\ell(n)$ the length of each $I \in \mathcal{L}_{\ell(n)}^{\mathcal{D}}$, where $\ell \in [n, n+2K_0].$ We also denote $\mathcal{L}_{\ell}^{\mathcal{D}}:=\mathcal{L}_{\ell(n)}^{\mathcal{D}}.$ Let $\A: \Sigma \to GL(d, \R)$ be a one-step cocycle.   Assume that $\A:\Sigma \to GL(d, \R)$ is a typical cocycle. By Corollary \ref{dominated-one-step-cocycle}, the one-step cocycle $\mathcal{B}: (\mathcal{L}_{\ell}^{\mathcal{D}})^{\Z} \to GL(d, \R)$ over a full shift $( (\mathcal{L}_{\ell}^{\mathcal{D}})^{\Z}, f )$ defined by $\mathcal{B}(\omega):=\A_{J_1(I)IJ_2(I)}$, where $\mathcal{B}$ depends only on the zero-th symbol $J_1(I)IJ_2(I)$ of $\omega \in (\mathcal{L}_{\ell}^{\mathcal{D}})^{\Z}$, is dominated (see Remark \ref{dominated-wedge}).  It is easy to see that $(\mathcal{L}_{\ell}^{\mathcal{D}})^{\Z} \subset \Sigma.$

We define a pressure on the dominated subsystem $\mathcal{L}_{\ell}^{\mathcal{D}}$ by setting
\[P_{\ell, \mathcal{D}}(\log\varphi):=\lim _{k \rightarrow \infty} \frac{1}{k} \log \sum_{I_{1}, \ldots, I_{k} \in \mathcal{L}_{\ell}^{\mathcal{D}}}\varphi(I_1 \ldots I_k),\]
where $\varphi: \mathcal{L} \rightarrow \mathbb{R}_{\geq 0}$ is submultiplicative, i.e.,
$$
\varphi(\mathrm{I}) \varphi(\mathrm{J}) \geq \varphi(\mathrm{IJ}).
$$
for all $I, J \in \mathcal{L}$ with $IJ \in \mathcal{L}.$

In the following proposition, the author \cite[Proposition 5.8]{Moh22} proved that if a cocycle $\mathfrak{A}$ is dominated with index $i$, which can be characterized in terms of the existence of invariant cone fields or multicones (see \cite{BGO, CP}), then $\{\log \|\mathfrak{A}^n\|\}_{n \in \mathbb{N}}$ is almost additive.
\begin{prop} \label{prop:add}
Let $X$ be a compact metric space, and let $\mathcal{A}: X\rightarrow GL(d, \R)$ be a matrix cocycle over a homeomorphism $(X,T)$. Assume that the cocycle $\mathcal{A}$ is dominated with index 1. Then, there exists $\kappa>0$ such that for every $m,n>0$ and for every $x\in X$ we have

\[
||\mathcal{A}^{m+n}(x)|| \geq \kappa ||\mathcal{A}^m(x)|| \cdot ||\mathcal{A}^n(T^m(x))||.
\]

\end{prop}

%The bounded distortion property of $\Phi_{\mathcal{A}}$ follows from the Hölder continuity
%and the fiber-bunching assumption on $\mathcal{A}_{t}$ that admits
%canonical holonomies; see \eqref{uniform_continuity}.

Let $\A:\Sigma \to GL(d, \R)$ be a one-step cocycle. Assume that $\mathcal{A}:\Sigma \to GL(d, \R)$ is a typical cocycle. Then, we can construct a dominated cocycle $\mathcal{B}$ as we explained above. We denote 
\[
\Psi(\mathcal{B}):=\left(\log \sigma_{1}(\mathcal{B}), \ldots, \log \sigma_{d}(\mathcal{B})\right).
\]

\begin{thm}\label{continuity_potential} Assume that $\left(A_1, \ldots, A_k\right) \in GL(d, \R)^k$ generates a one-step cocycle $\mathcal{A}:\Sigma \to GL(d, \R).$ Suppose that $\mathcal{A}:\Sigma \to GL(d, \R)$ is a typical cocycle. Then, 
\[\lim _{\ell \rightarrow \infty} \frac{1}{\ell} P_{\ell, \mathcal{D}}(\langle q, \Psi(\mathcal{B}) \rangle )=P(\log \psi^{q}(\mathcal{A})),\]
uniformly for all $q$ on any compact subsets of $\R^d$.
\end{thm}

\begin{proof}

First, we prove the lower bound. Since the matrix cocycle $\mathcal{B}$ is dominated, there is $\kappa>0$ such that for any $n \in \mathbb{N}$, $I, J \in \bigcup_{k=1}^{\infty}\left(\mathcal{L}_{\ell(n)}^{\mathcal{D}}\right)^{k}$, and  $t \in \{1, \ldots, d-1\}$,
 we have
\[\|\mathcal{B}^{\wedge t}_{IJ}\| \geq \kappa \|\mathcal{B}^{\wedge t}_{I}\| \|\mathcal{B}^{\wedge t}_{J}\| \]
 by proposition \ref{prop:add}. 
 Therefore, 
 $$
\begin{aligned}
 \frac{1}{\ell} P_{\ell, \mathcal{D}}(\langle q,\Psi(\mathcal{B}) \rangle)= & \lim _{k \rightarrow \infty} \frac{1}{\ell k} \log \sum_{I_1, \ldots, I_k \in \mathcal{L}_{\ell}^{\mathcal{D}}} e^{\langle q, \Psi(\mathcal{B}_{I_1 \ldots I_k}) \rangle} \\
& \geq \lim_{k \rightarrow \infty} \frac{1}{\ell k} \log \sum_{I_{1}, \ldots, I_{k} \in \mathcal{L}_{\ell}^{\mathcal{D}}}e^{\sum_{i=1}^{k}\langle q, \Psi(\mathcal{B}_{I_{i}}) \rangle +C(\kappa) k}  \\\\
&=\underbrace{\frac{1}{\ell}\log \sum_{I' \in \mathcal{L}_{\ell}^{\mathcal{D}}} e^{\langle q, \Psi(\mathcal{B}_{I'}) \rangle}+\frac{C(\kappa)}{\ell}}_\text{(1)}.
\end{aligned}
$$
Assume that $\omega \in (\mathcal{L}_{\ell}^{\mathcal{D}})^{\Z}$ such that $I'=J_1 I J_2$ is the zero-th symbol of $\omega$, where  $I \in \mathcal{L}_n$ and $|J_i| \leq K_0$ (see Corollary \ref{dominated-one-step-cocycle}). Thus,  for any $i \in \{2, \ldots, d\}$,
\begin{gather}
 \sigma_{i}(\mathcal{B}_{I'})= \sigma_{i}(\mathcal{B}(\omega))=\frac{\|\mathcal{B}^{\wedge i}(\omega)\|}{\|\mathcal{B}^{\wedge i-1}(\omega)\|}\stepcounter{equation}\tag{\theequation}\label{myeq2}\\
   \hspace{3cm} =\frac{\|\A^{\wedge i} _{J_1 I J_2}\|}{\|\A^{\wedge i-1} _{J_1 I J_2}\|}.
\end{gather}

By Corollary \ref{dominated-one-step-cocycle}, there are $K_1, K_2 \in \N$ such that for any $t\in \{1, \ldots, d\},$
\begin{equation}\label{upper-bound-B}
\|\A^{\wedge t}_{J_1 I J_2}\| \leq K_1 K_2 \|\A^{\wedge t}_{I}\|.
\end{equation}

Similarly, by Corollary \ref{dominated-one-step-cocycle}, there are $K_3, K_4 \in \N$ such that for any $t\in \{1, \ldots, d\},$
\begin{equation}\label{lower-bound-B}
\|\A^{\wedge t}_{J_1 I J_2}\| \geq K_3 K_4 \|\A^{\wedge t}_{I}\|.
\end{equation}

Then, by \eqref{myeq2}, \eqref{upper-bound-B}, and \eqref{lower-bound-B}, 
\[
(1)  \geq \frac{n}{\ell}\frac{1}{n} \log \sum_{I\in \mathcal{L}_{n}} \psi^{q}(\mathcal{A}_{I})+\frac{C(\kappa)}{\ell}+\frac{C'}{\ell},\]
where $C':=C'(K_0, K_1, K_2, K_3, K_4).$ Then, we see that  
\[\lim _{\ell \rightarrow \infty} \frac{1}{\ell} P_{\ell, \mathcal{D}}(\langle q, \Psi(\mathcal{B}) \rangle ) \geq P(\log \psi^{q}(\mathcal{A})),\]
 by taking $\ell \rightarrow \infty$, which implies $n \rightarrow \infty$.

 Now, we will prove the other side. Since
 $$
\begin{aligned}
 P(\log \psi^q(\mathcal{A}))= & \lim _{m \rightarrow \infty} \frac{1}{m} \log \sum_{I \in \mathcal{L}_{m}}   \psi^{q}(\mathcal{A}_{I})\\
 & = \lim _{k \rightarrow \infty} \frac{1}{(n+2K_0) k} \log \sum_{I_1, \ldots, I_k \in \mathcal{L}_{n+2K_0}}    \psi^{q}(\mathcal{A}_{I_1 \ldots I_k})\\
 & \geq  \lim _{k \rightarrow \infty} \frac{\ell}{(n+2K_0) k}\frac{1}{\ell} \log \sum_{I_1, \ldots, I_k \in \mathcal{L}_{\ell}^{\mathcal{D}}}   \psi^{q}(\mathcal{A}_{I_1 \ldots I_k})\\
& =  \lim _{k \rightarrow \infty} \frac{\ell}{(n+2K_0) k}\frac{1}{\ell} \log \sum_{I_1, \ldots, I_k \in \mathcal{L}_{\ell}^{\mathcal{D}}} e^{\langle q,\Psi(\mathcal{B}_{I_1 \ldots I_k}) \rangle},
 \end{aligned}
 $$

 The statement follows by taking $\ell \to \infty.$
\end{proof}

\section{Multifractal formalism for almost additive potentials}

For $i=1, \ldots, d$, let $\Phi_{i}:=\{\log \phi_{n}^{i}\}_{n=1}^{\infty}$ be an almost additive potential over a topologically mixing subshift of finite type $(\Sigma, T)$. We denote $\bold{\Phi}:=\left(\Phi_{1}, \ldots, \Phi_{d}\right)$.  Let $\mu \in \mathcal{M}(\Sigma, T)$ . We denote 
\[\chi(\mu, \bold{\Phi}):=\left(\chi(\mu, \Phi_{1}), \ldots, \chi(\mu, \Phi_{d})\right).\]

We say that $\bold{\Phi}$ has bounded distortion property:  There is $C>0$ such that for any $n\in \N$ and $I \in \mathcal{L}_{n}$, we have
\begin{equation}\label{bounded-dist}
C^{-1} \leq\frac{\phi_{n}^{i}(x)}{\phi_{n}^{i}(y)}\leq C
\end{equation} 
for all $x, y \in [I]$ and $i\in \{1, \ldots, d\}.$

\begin{thm}\label{topological pressure_almost additive}
Let $(\Sigma, T)$ be a topologically mixing subshift of finite type. Denote $\bold{\Phi}:=\left(\Phi_{1}, \ldots, \Phi_{d} \right)$, where $\Phi_{i}:=\{\log \phi_{n}^{i}\}_{n=1}^{\infty}$ is an almost additive potential and $\bold{\Phi}$  has bounded distortion property. Then,  for any $q\in \R^d,$ there is a unique ergodic Gibbs equilibrium measure $\mu_{q}$ for $\langle q, \bold{\Phi} \rangle $ such that 
 \[D_{q}P(\langle q, \bold{\Phi} \rangle)=\chi(\mu_{q}, \bold{\Phi}).\]

\end{thm}
\begin{proof}
This follows from the combination of \cite[Theorem 3.3]{FH} and \cite[Theorem 10.1.9]{Bar}.
\end{proof}

\begin{thm}\label{convergence-almost-additive}
Suppose $\{\nu_{n}\}_{n=1}^{\infty}$ is a sequence in $\mathcal{M}(X)$ and $\Phi=\{\log \phi_{n}\}_{n=1}^{\infty}$ is an almost additive potential over a topological dynamical system $(X,T).$  We form the new sequence $\{\mu_{n}\}_{n=1}^{\infty}$ by $\mu_{n}=\frac{1}{n}\sum_{i=0}^{n-1}\nu_{n}oT^{i}$. Assume that $\mu_{n_{i}}$ converges to $\mu$ in $\mathcal{M}(X)$ for some subsequence $\{n_i\}$ of natural numbers.  Then $\mu$ is an $T$-invariant measure and 
\[
\lim_{i \rightarrow \infty} \frac{1}{n_{i}}\int \log \phi_{n_{i}}(x)d\nu_{n_i}(x)= \chi (\mu, \Phi).
\]

\end{thm}
\begin{proof}
It follows from \cite[Lemma A.4]{FH}.
\end{proof}

Let a matrix cocycle $\mathcal{A} :\Sigma \to GL(d, \R)$  over a topologically mixing subshift of finite type $(\Sigma, T)$  be dominated.  We recall that
\[
\Phi_{\mathcal{A}}= \left(\log \sigma_{1}(\mathcal{A}), \ldots, \log \sigma_{d}(\mathcal{A})\right).
\]
By Proposition \ref{prop:add}, for each $i \in \{1, \ldots, d\}$, $\{\log \sigma_{i}(\mathcal{A}^n)\}_{n=1}^{\infty}$ is almost additive. We also say that  $\Phi_{\mathcal{A}}$ has bounded distortion property if $\{\log \sigma_i(\mathcal{A}^n)\}_{n=1}^{\infty}$ satisfies \eqref{bounded-dist} for all $i\in \{1, \ldots, d\}.$

For $\vec{\alpha}\in \R^{d}$, we recall the level set
\[E(\vec{\alpha})=\bigg\{ x\in \Sigma: \lim_{n \to \infty}\frac{1}{n}\log \sigma_{i}(\mathcal{A}^{n}(x))=\alpha_i \text{ for }i=1,2, \ldots, d \bigg\}.\]

We also recall the Lyapunov spectrum

\[\vec{L}=\bigg\{\vec{\alpha} \in \R^{d}: \exists x \in \Sigma \text{ such that } \lim_{n\to \infty} \frac{1}{n}\log \sigma_{i}(\mathcal{A}^{n}(x))=\alpha_i \text{ for }i=1,2, \ldots, d \bigg\}.\]

\begin{thm}\label{lower-bound} Assume that the matrix cocycle $\mathcal{A}:\Sigma \to GL(d, \R)$ over a topologically mixing subshift of finite type $(\Sigma, T)$ is dominated and $\Phi_{\mathcal{A}}$ has bounded distortion property. If $\alpha \in \mathring{\vec{L}}$, then
\[h_{\mathrm{top}}(E(\vec{\alpha}))\geq \inf _{q \in \mathbb{R}^{d}}\left\{P\left(\langle q , \Phi_{\mathcal{A}}-\vec{\alpha} \rangle \right) \right\} .\]
\end{thm}
\begin{proof}
The proof is similar to \cite[Lemma 12.1.6]{Bar}. We give a proof here for the convenience of the readers.

 Let $r>0$ be the distance of $\vec{\alpha}$ to $\mathbb{R}^{d} \backslash \vec{L}$. Take $q \in \mathbb{R}^{d}$ and define:
\[ F(q):=P\left(\langle q, \Phi_{\mathcal{A}} -\vec{\alpha}  \rangle \right).\]
Given $\beta=\left(\beta_{1}, \ldots, \beta_{d}\right) \in \mathbb{R}^{d}$ with $\beta_{i}=\alpha_{i}+r \operatorname{sgn} q_{i} /(2 d)$ for each $i$, we have:
$$
\|\beta-\vec{\alpha}\|=\sum_{i=1}^{d}\left|\beta_{i}-\alpha_{i}\right|=\sum_{i=1}^{d}\left|\frac{1}{2 d} r \operatorname{sgn} q_{i}\right|=\frac{r}{2}<r
$$
and hence $\beta \in \vec{L}$. Note that $\{\log \sigma_i (\mathcal{A}^n)\}_{n=1}^{\infty}$ is almost additive by Proposition \ref{prop:add}. Take $x \in E(\beta)$. Denote $\mu_{x, n}:=\frac{1}{n} \sum_{i=1}^{n-1} \delta_{T^{i}}{ }_{x}$. Then there exists $n_{j} \uparrow \infty$ so that $\mu_{x, n_{j}} \rightarrow \mu$ for some $\mu \in \mathcal{M}(\Sigma, T)$. Apply Theorem \ref{convergence-almost-additive} (in which we take $\nu_{n}=\delta_{x}$ ) to obtain, 
$$
\lim _{j \rightarrow \infty} \frac{1}{n_j} \log \sigma_{i}(\mathcal{A}^{n_j}(x))=\lim _{n \rightarrow \infty} \frac{1}{n} \int \log \sigma_{i}(\mathcal{A}^{n}(x)) d \mu.
$$

Since $x \in E(\beta),$ 
$$
\lim _{n \rightarrow \infty} \frac{1}{n} \int \log \sigma_{i}(\mathcal{A}^{n}(x)) d \mu=\beta_i.
$$

% Assume that $E(\vec{\alpha}) \neq \emptyset$ for some $\vec{\alpha}=\left(a_{1}, \ldots, a_{k}\right) \in \mathbb{R}^{d}$. Take $x \in E(\vec{\alpha})$. Denote $\mu_{x, n}=(1 / n) \sum_{j=1}^{n-1} \delta_{T^{j} x}$. Then there exists $n_{j} \uparrow \infty$ so that $\mu_{x, n_{j}} \rightarrow \mu$ for some $\mu \in \mathcal{M}(X, T)$. Apply Theorem \ref{convergence-almost-additive} (in which we take $\nu_{n}=\delta_{x}$ ) to obtain
 %\[ \lim_{n\to \infty}\frac{1}{n}\log \sigma_{i}(\mathcal{B}^n(x))=\lim_{n\to \infty}\frac{1}{n} \int \log \sigma_{i}(\mathcal{B}^n(x)) d\mu(x).\]
%We now work with the pressure for $T^{n}$ which we denote by $P_{n}$. Note that we have $P_{n}\left(S_{n} \cdot\right)=n P(\cdot)$.

Denote $\Phi_{n}:=\Phi_{\mathcal{A}^n}.$ By variational principle (see \eqref{varitional}), we have
\[
F(q) \geqslant h_{\mu}(T)+\lim_{n \rightarrow \infty} \frac{1}{n}\left\langle q, \int \left(\Phi_{n}-n \vec{\alpha}  \right) d \mu \right\rangle.
\]
Since
$$
\begin{aligned}
\left\langle q, \int (\beta-\vec{\alpha}) d \mu\right\rangle &=\sum_{i=1}^{d} q_{i} \int \left(\beta_{i}-\alpha_{i}\right) d \mu=\sum_{i=1}^{d} \frac{1}{2 d} r q_{i} \operatorname{sgn} q_{i},\\
\end{aligned}
$$
we have:
$$
\frac{1}{n}\left\langle q, \int \left(\Phi_{n}-n\vec{\alpha} \right) d \mu\right\rangle=\frac{1}{2 d} r \sum_{i=1}^{d}\left|q_{i}\right| +\frac{1}{n}\left\langle q, \int \left(\Phi_{n}-n\beta \right) d \mu\right\rangle .
$$
Taking the limit when $n \rightarrow \infty$, we obtain:
\[ \lim_{n\to \infty} \frac{1}{n}\left\langle q, \int \left(\Phi_{n}-n\vec{\alpha} \right) d \mu\right\rangle \geq \frac{1}{2 d} r \|q\|.\]

Since $h_{\mu}(T) \geq 0,$ 
\begin{equation}\label{One-side-pressure}
F(q) \geq \frac{1}{2 d} r \|q\|.
\end{equation}
We note that the right-hand side of inequality \eqref{One-side-pressure} takes arbitrarily large values for $\|q\|$ sufficiently large. Thus, there exists $R \in \mathbb{R}$ such that $F(q) \geqslant F(0)$ for every $q \in \mathbb{R}^{d}$ with $\|q\| \geqslant R$.

Since $F$ is differentiable by Theorem \ref{topological pressure_almost additive}, it attains a minimum at some point $q=q(\vec{\alpha})$ with $\|q(\vec{\alpha})\| \leqslant R$, thus satisfying $D F(q(\vec{\alpha}))=0$. By Theorem \ref{topological pressure_almost additive}, there is an ergodic Gibbs equilibrium measure $\mu_{q(\vec{\alpha})}$ of the sequence of functions
$$
\langle q(\vec{\alpha}), \Phi_{n}-\vec{\alpha} \rangle.
$$

Moreover, by Theorem \ref{topological pressure_almost additive}, \begin{equation}\label{LE-equ}
 D F(q(\vec{\alpha}))=\lim_{n\to \infty}\frac{1}{n} \int (\Phi_{n}-n\vec{\alpha} ) d\mu_{q(\vec{\alpha})}=\vec{0}.
\end{equation}

Taking into account the above observations, 
 \[\min _{q \in \mathbb{R}^{d}} \{P\left(\langle q, \Phi_{\mathcal{A}}-\vec{\alpha} \rangle \right)\}=F(q(\vec{\alpha}))= h_{\mu_{q(\vec{\alpha})}}(T).\]

Note that by  \eqref{LE-equ}, $\mu_{q(\vec{\alpha})}(E(\vec{\alpha}))=1.$ Therefore,
\[h_{\mathrm{top}}(E(\vec{\alpha}))\geq h_{\mu_{q(\vec{\alpha})}}(T) \geq \inf_{q \in \mathbb{R}^{d}} \{P\left(\langle q, \Phi_{\mathcal{A}}-\vec{\alpha} \rangle \right)\} .\]

 \end{proof}

\begin{thm}\label{dominated_case}
 Assume that $\left(A_1, \ldots, A_k\right) \in GL(d, \R)^k$ generates a one-step cocycle $\mathcal{A}:\Sigma \to GL(d, \R).$ Suppose that the one-step cocycle $\mathcal{A}:\Sigma \to GL(d, \R)$ is dominated. Then,
\[h_{\mathrm{top}}\left( E(\vec{\alpha})\right) =  \inf _{q \in \mathbb{R}^{d}}\left\{P\left(\langle q, \Phi_{\mathcal{A}} \rangle\right)- \langle q , \vec{\alpha} \rangle \right\}
\]
for all $\vec{\alpha} \in \mathring{\vec{L}}$.

\end{thm}
\begin{proof}
The lower bound follows from Theorem \ref{lower-bound} and the upper bound follows from Theorem \ref{HD-one side}.
\end{proof}

\section{The proof of Theorem A}

Let $\mathcal{A}:\Sigma \to GL(d, \R)$ be a one-step typical cocycle. Then, we can construct a
dominated cocycle $\mathcal{B}: (\mathcal{L}_{\ell(n)}^{\mathcal{D}})^{\Z} \to GL(d, \R)$ as we explained in Section \ref{dominated_subsystem}.  By Proposition \ref{prop:add}, there is $C>0$ such that for all $n, k \in \N$ and $x \in (\mathcal{L}_{\ell(n)}^{\mathcal{D}})^{\Z}$,
\[\left|\log \sigma_{i}(\mathcal{B}^{k}(x))-S_{k}\log \sigma_{i}(\mathcal{B}(x)) \right|\leq C,\]
for $i \in \{1, \ldots, d\}$.

\begin{proof}[Proof of Theorem A]
We denote  
\[ E^{\ell, \mathcal{D}}(\vec{\alpha}):=\bigg\{x\in (\mathcal{L}_{\ell}^{\mathcal{D}})^{\Z}: \lim_{k\to \infty}\frac{1}{k}S_{k} \Psi(\mathcal{B}(x))= \vec{\alpha}\bigg\}.\]

It is easy to see that 

\begin{equation}\label{eq1}
\frac{1}{n+2K_0} h_{\mathrm{top}}(E^{\ell, \mathcal{D}}( \vec{\alpha}))\leq h_{\mathrm{top}}(E(\vec{\alpha})).
\end{equation}

We denote
$$
\begin{aligned}
&s_{0}(\vec{\alpha}):= \inf_{q\in \R^{d}}\{P(\log \psi^{q}(\mathcal{A}))- \langle q, \vec{\alpha} \rangle \},\\
&s_{\ell}(\vec{\alpha}):=\inf_{q\in \R^{d}}\{P_{\ell ,\mathcal{D}}(\langle q,  \Psi(\mathcal{B}) \rangle)- \langle q, \ell\vec{\alpha} \rangle\},
\end{aligned}
$$
where $\ell=\ell(n)\in [n, n+2K_0].$ By Theorem \ref{continuity_potential},
 \begin{equation}\label{continuity_transfor}
\frac{1}{\ell}s_{\ell}(\vec{\alpha})\to s_{0}(\vec{\alpha}).
\end{equation}

Note that by Corollary \ref{proof-upper-bound}, Theorem \ref{dominated_case} and \eqref{eq1},
$$
\begin{aligned}
\frac{1}{\ell}s_{\ell}(\vec{\alpha})&=\frac{n+2K_0}{n+2K_0}\frac{1}{\ell}h_{\mathrm{top}}(E^{\ell, \mathcal{D}}( \vec{\alpha}))\\
&\leq \frac{n+2K_0}{\ell} h_{\mathrm{top}}(E(\vec{\alpha}))\\
&\leq \frac{n+2K_0}{\ell} s_{0}(\vec{\alpha}).
\end{aligned}
$$
Therefore, by \eqref{continuity_transfor}, we see that
\[h_{\mathrm{top}}(E(\vec{\alpha}))=s_{0}(\vec{\alpha}),\]
when $\ell \to \infty.$
\end{proof}

\section*{\textbf{Declarations}}

\subsubsection*{\textbf{Conflict of interest}} The author declares that he has no conflict of interest.

\subsubsection*{\textbf{Funding}} This work was supported by the Knut and Alice Wallenberg Foundation.

\subsubsection*{\textbf{Ethical Approval}} This article does not contain any studies with human participants performed by any of the
authors.

\subsubsection*{\textbf{Data Availability}} Data sharing not applicable to this article.

\bibliographystyle{acm}
\bibliography{Entropy-spectrum}
\end{document}